\documentclass[reqno,11pt]{amsart}
\usepackage[utf8]{inputenc}

\usepackage{amssymb}
\usepackage{amsgen}
\usepackage{amsmath}
\usepackage{amsthm}
\usepackage{mathrsfs}
\usepackage{cite}
\usepackage{amsfonts}
\usepackage{tikz-cd}
\usepackage{tikz}

\usepackage{centernot}
\usepackage{mathtools}
\usepackage{ stmaryrd }
\usepackage{comment}
\usepackage{enumitem}
\usepackage{url}
\usepackage{hyperref}
\usepackage{soul}

\DeclareMathOperator\supp{supp}

\newtheorem{Thm}{Theorem}[section]
\newtheorem{Prop}[Thm]{Proposition}

\newtheorem{Lemma}[Thm]{Lemma}
{\theoremstyle{definition} 
\newtheorem{Def}[Thm]{Definition}}
{\theoremstyle{remark}
\newtheorem{Rmk}[Thm]{Remark}}
\newtheorem{Cor}[Thm]{Corollary}

\numberwithin{equation}{section}

\title{Categorical Chain Conditions for Etale Groupoid Algebras}
\author{Sunil Philip}
\date{\today}

\begin{document}

\maketitle

\section{Introduction}
In the early 1960s, William G. Leavitt introduced Leavitt algebras as examples of rings that violate the invariant basis number property and are universal for the property that $R^m \cong R^{m+k}$ with $k$ minimal for any $k, m \in \mathbb{N}_{>0}$.  Additionally, Leavitt algebras are also finitely presented and simple when $m=1$. \cite{ LeavittModTypeRing, LeavittModTypeHomImg, Abramssurvey}.
\newline
In the late 1970s, in a seemingly unrelated development, Cuntz \cite{cuntz}, unaware of Leavitt's work,  defined universal $C^{\ast}$-algebras $\mathcal{O}_n$ having the same presentation as the Leavitt algebra when $m=1$ and $k=n-1$.  Cuntz algebras were examples of separable $C^{\ast}$-algebras that are simple and  purely infinite.  By 1980, Cuntz and Krieger, motivated by explorations in symbolic dynamics, generalized Cuntz algebras to a class of $C^{\ast}$-algebras associated to binary-valued matrices (equivalently simple directed graphs) \cite{cuntzkrieger}.  Around this same time, in groundbreaking work, Renault \cite{Renault, RenaultThesis} worked out the basic theory of groupoid $C^{\ast}$-algebras.  By the late 1990s, Kumjian, Pask, Tomforde, Raeburn, and Renault \cite{KumjianPaskRaeburnRenaultGraphsGrpdsCalg} defined a Hausdorff ample groupoid from certain classes of directed graphs, which could then be used to construct $C^{\ast}$-algebras.  These graph groupoid/graph $C^{\ast}$-algebras universally satisfy the Cuntz-Krieger relations \cite{cuntzkrieger}.   The original Cuntz algebras $\mathcal{O}_n$ \cite{Renault}, Cuntz-Krieger algebras, and graph $C^{\ast}$-algebras, can be interpreted as groupoid $C^{\ast}$-algebras and were, therefore, open to groupoid techniques in addition to more traditional methods of studying these $C^{\ast}$-algebras.  In fact, many $C^{\ast}$-algebraic properties of graph $C^{\ast}$-algebras can be characterized by the structural properties of the underlying graph.  Also in \cite{KumjPaskHighRnkGrph} a higher dimensional graph called a $k$-graph was developed with an associated $C^{\ast}$-algebra.  
\newline
In the early 2000s, Leavitt algebras based on paths in a graph were defined and dubbed Leavitt path algebras \cite{AraGoodearlPardoK0, NstableKGrphAlg}.  Since the defining relations of the Leavitt path algebras are the same as those of the graph $C^{\ast}$-algebras, it would be fair to assume there would be some connection between the purely algebraic Leavitt path algebras and their analytic $C^{\ast}$-algebra cousins.  It has been shown that when the coefficient ring of the Leavitt path algebra is the field $\mathbb{C}$, the corresponding $C^{\ast}$-algebra is the universal norm completion of the Leavitt path algebra.  In fact, when the graph is finite, the graph $C^{\ast}$-algebra is purely infinite simple if and only if the corresponding Leavitt path algebra is purely infinite simple.  For finite graphs, the graph $C^{\ast}$-algebra has the same $K_0$-group as the corresponding Leavitt path algebra \cite{AraGoodearlPardoK0, groupoidapproachleavitt,LeavittBook}.
\newline
In \cite{SteinbergGroupoidAlgebra}, Steinberg defined a groupoid $R$-algebra as a convolution algebra of the $R$-module generated by characteristic functions of certain types of compact open subsets of an ample groupoid, where $R$ is any unital commutative ring.  Ample groupoids are a special type of etale groupoid.  Working independently, Clark, Farthing, Sims, and Tomforde \cite{ operatorguys1} developed the same groupoid algebra for the specific case where $R=\mathbb{C}$ and the groupoid is Hausdorff.  Many in the research community refer to these algebras as Steinberg algebras and etale groupoid algebras.  We will follow the latter convention.  The etale groupoid algebras include many important classes of rings such as:
\begin{enumerate}
    \item group algebras,
    \item commutative algebras over a field generated by idempotents, 
    \item crossed products of rings described in (1) and (2),
    \item Leavitt path algebras,
    \item higher rank graph algebras, and
    \item inverse semigroup algebras
\end{enumerate}
Steinberg \cite{gpdchain} characterized the classical chain conditions for etale groupoid algebras.  He showed that if the classical chain conditions hold, then the etale groupoid algebra is unital and the groupoid has finitely many objects.  So the classical chain conditions are quite restrictive for etale groupoid algebras.  In this paper, we will explore generalized (categorical and local) chain conditions for non-unital etale groupoid algebras.  In order to get out of the unital case we are forced into looking at these generalized chain conditions, because a ring with local units can only satisfy the classical chain conditions when the ring is unital.  We will show that these chain conditions are equivalent to the topological property of discreteness of the underlying groupoid, classical chain conditions on the group rings of the isotropy groups, and a matrix algebra decomposition based on the orbits of the groupoid. Our second major result is an analogous result characterizing semisimplicity of etale groupoid algebras. 
\newline 
Our main results, Theorems \ref{thm:1001} and \ref{thm:1012}, recover \cite[Theorems 4.2.7 and 4.2.12]{LeavittBook} which are the definitive version of earlier results of \cite[Theorems 2.4 and 3.7]{AbramsetalchaincondLPA}.  Additionally, our analogs of these results hold when the coefficient ring is an arbitrary unital commutative ring, not just a field.
\section{Groupoid Algebras}
\subsection{Topological, etale, and ample groupoids}
A groupoid $\mathscr{G}$ is a small category where every arrow is an isomorphism.  We view $\mathscr{G}$ as a set with a partially defined multiplication.  We identify objects with appropriate units and call the set of units $\mathscr{G}^{(0)}$.   A topological groupoid $\mathscr{G}$ is a groupoid with a topology making multiplication and inversion maps continuous.   In the setting of a topological groupoid, we endow $\mathscr{G}^{(0)}$ with the subspace topology.   Groupoids are endowed with basic structure maps 
$$
d,r: \mathscr{G} \longrightarrow \mathscr{G}^{(0)}
$$
$$
g \xmapsto{\ \ \ d \ \ \ } g^{-1}g
$$
$$
g \xmapsto{\ \ \ r \ \ \ } gg^{-1}
$$
In this paper we assume that our topological groupoids are \textit{locally compact}.  In our setting, a \textit{locally compact groupoid}, is a groupoid that is \textit{locally compact} and whose unit space is locally compact Hausdorff in the induced topology.  Etale groupoids are topological groupoids where the domain $d:\mathscr{G} \rightarrow \mathscr{G}^{(0)}$ and range $r:\mathscr{G} \rightarrow \mathscr{G}^{(0)}$ are local homeomorphisms.  An \textit{open bisection} is an open subset $U$ of $\mathscr{G}$ such that $d {\restriction}_U: U \rightarrow d (U)$ and $r {\restriction}_U: U \rightarrow r (U)$  are injective, and therefore  homeomorphisms, onto open subsets of $\mathscr{G}^{0)}$.  Hence open bisections are always Hausdorff.  Ample groupoids are etale groupoids with a Hausdorff unit space and a basis of compact open bisections, equivalently $\mathscr{G}^{(0)}$ has a basis of compact open sets.  Note that a discrete groupoid is always ample.
\begin{Rmk} \label{rmk:0007}
For etale groupoids, since $d, r :\mathscr{G} \rightarrow \mathscr{G}^{(0)}$ are local homeomorphisms, $\mathscr{G}$ is discrete if and only if $\mathscr{G}^{(0)}$ is discrete. 
\end{Rmk}
For any set $U \subseteq \mathscr{G}^{(0)}$, $\mathscr{G}_U = d^{-1}(U)$, $\mathscr{G}^U = r^{-1}(U)$, and $\mathscr{G}{\restriction}_U = \mathscr{G}^U_U = d^{-1}(U) \cap r^{-1}(U)$.  We say that  $U$ is an \textit{invariant subset} of $\mathscr{G}^{(0)}$ if $\mathscr{G}_U = \mathscr{G}{\restriction}_U$.
\begin{Def}[Isotropy Group and Orbit] \label{def:0003} If $x \in \mathscr{G}^{(0)}$, then $\mathscr{G}_x^x = \{ g \in \mathscr{G} : d(g)=x =r(g) \}$ is the \textit{isotropy group} at $x$.  The \textit{orbit} $\mathcal{O}_x$ of $x \in \mathscr{G}^{(0)}=r(\mathscr{G}_x)$ is the set of all $y \in \mathcal{G}^{(0)}$ such that there exists an arrow $g:x \rightarrow y$.  
\end{Def}
\begin{Rmk} \label{rmk:0001}
The isotropy groups of elements in the same orbit are isomorphic and $\mathcal{O}_x$ is the smallest invariant subset containing $x$.  Let $\mathscr{G}$ be a groupoid, $U$ is an invariant subset of $\mathscr{G}^{(0)}$ if and only if $U$ is a union of orbits.
\end{Rmk}
\subsection{Etale Groupoid Algebras}
Let $\mathscr{G}$ be an ample groupoid  and $R$ a commutative ring with unit.  We define $R \mathscr{G}$ to be the $R$-submodule of $R^{\mathscr{G}}$ spanned by the characteristic functions $\chi_U$ where $U$ is a compact open bisection.  Then this $R$-module becomes an $R$-algebra by defining a convolution product on $R\mathscr{G}$.  
For all $f_1, f_2 \in R\mathscr{G}$, the convolution product is defined as 
\begin{equation} 
    f_1 \ast f_2 = \sum_{d(h)=d(g)} f_1(gh^{-1}) f_2(h).
  \end{equation}
Groupoid algebras have an involution, given by $f^{\ast}(g) = f(g^{-1})$ for all $f \in R \mathscr{G}$ and $g \in \mathscr{G}$.  Further details on the construction of etale groupoid algebras can be found in \cite{SteinbergGroupoidAlgebra}.  Since groups are single object groupoids, group rings are a special case of groupoid algebras.
\begin{Prop} [Proposition 2.1 of \cite{groupoidbundles}] \label{prop:1004}
Suppose $\mathcal{G}$ is an ample groupoid and $R$ a commutative ring with unit. Let $\mathcal{B}$ denote the generalized boolean algebra of compact open subsets of $\mathcal{G}^{(0)}$.
\begin{enumerate}
    \item $\mathcal{B}$ is directed.
    \item If $U \in \mathcal{B}$, then $\mathcal{G}{\restriction}_U$ is an open ample subgroupoid of $\mathcal{G}$.
    \item \label{prop:1004_3} If $U \in \mathcal{B}$, then $\chi_U R \mathcal{G} \chi_U \cong R \mathcal{G}{\restriction}_U$
    \item \label{prop:1004_4} $R \mathcal{G} = \bigcup_{U \in \mathcal{B}} \chi_U R \mathcal{G} \chi_U = \underset{ U \in \mathcal{B}}{\varinjlim} R \mathcal{G}{\restriction}_U$
\end{enumerate}
\end{Prop}
We will need the following lemma, which is a well-known result about Boolean algebras and can be found in \cite{gpdchain}
\begin{Lemma} [Lemma 3 of \cite{gpdchain}] \label{thm:0001} Let $X$ be a Hausdorff space with a basis of compact open sets.  Then $X$ satisfies the ascending chain condition on compact open subsets if and only if $X$ is finite and hence discrete.
\end{Lemma}
\begin{proof}
($\Leftarrow$)  If $X$ is finite, the result is obvious.
\newline
($\Rightarrow$)  Assume $X$ satisfies the ascending chain condition on compact open subsets.  Then $X$ must be compact for $X$ must have a maximal compact open subset $K$.  If $K \ne X$ and $x \in X \setminus K$, then there exists a compact open neighborhood $U$ of $x$ and $K \cup U \supsetneq K$. But $K \cup U$ is a compact open subset of $X$ strictly larger than $K$, so $X = K$.  We can appeal to the fact that $X$ is compact Hausdorff to see that compact open subsets of $X$ are closed under complement; therefore, $X$ also satisfies the descending chain condition on compact open subsets.  So each $x \in X$ is contained in a minimal compact open set $K_x$.  If $K_x \ne \{ x \}$, then since $X$ is Hausdorff with a basis of compact open subsets, for any $y \in K_x \setminus \{ x \}$, there is a compact open subset $V$ with $x \in V \subseteq K_x$ and $y \notin V$.  Since $x \in V \subsetneq K_x$, this contradicts $K_x$ being minimal.  So $K_x = \{ x \}$ which gives us that $X$ is discrete and compact and therefore finite.
\end{proof}
\begin{Prop} \label{thm:0009}
If $V, W \subseteq \mathscr{G}^{(0)}$, then $V \subseteq W$ if and only if $\chi_V \ast \chi_W = \chi_V$.
\end{Prop}
\begin{proof}
Recall from \cite{SteinbergGroupoidAlgebra} and \cite{groupoidapproachleavitt}, for any compact open bisections $V, W$, $\chi_V \ast \chi_W = \chi_{VW}$.  Suppose $V,W$ are any two compact subsets of $\mathscr{G}^{(0)}$, then $V,W$ are also compact open bisections and we have 
\begin{equation} \label{eq:0001}
\chi_V \ast \chi_W = \chi_{VW} =\chi_{V \cap W} 
\end{equation}
\newline
and $\chi_{V \cap W} = \chi_V$ if and only if $V \cap W = V$ if and only if $V \subseteq W$.
\end{proof}
\begin{Lemma} \label{thm:0003}
Suppose $R$ is a commutative ring with unit and $\mathscr{G}$ is an ample groupoid and $\mathscr{G}^{(0)} = \coprod _{\alpha \in J} U_{\alpha}$ where each $U_{\alpha}$ is a clopen invariant subset of $\mathscr{G}^{(0)}$, then $R \mathscr{G} =\bigoplus _{\alpha \in J} R \mathscr{G}{\restriction}_{U_{\alpha}} = \bigoplus _{\alpha \in J} R \mathscr{G}_{U_{\alpha}}$
\end{Lemma}
\begin{proof}
Recall that when $U_\alpha$ is an invariant subset of a groupoid $\mathscr{G}{\restriction}_{U_\alpha}$ = $\mathscr{G}_{U_\alpha}$, so the second equality is obvious.  Let $\mathscr{G}$ be an ample groupoid partitioned by full subgroupoids $\mathscr{G}_{U_{\alpha}}$ where each $U_{\alpha}$ is a clopen invariant subset of $\mathscr{G}^{(0)}$.  For an arbitrary $f \in R \mathscr{G}$, we can write $f= \sum_{i=1}^m r_i \chi_{V_i}$ where $m$ is finite and each $V_i$ is a compact open bisection.  Since $U_\alpha$ is a clopen invariant set,   $\mathscr{G}_{U_\alpha} =d^{-1}({U_\alpha} )$ is clopen.  So $V_i \cap \mathscr{G}_{U_\alpha}$ is a clopen subset of $V_i$ and hence also a compact open bisection for each $i = 1, \hdots, n$.
Then $\{ V_i \cap  \mathscr{G}_{U_{\alpha}}\}_{\alpha \in J}$ is a cover of $V_i$ by compact open bisections for each compact open bisection $V_i, i = 1, \hdots m$, and we can then find a finite subcover for each $V_i$.  We will refer to each subcover as $\mathscr{C}_i, i = 1, \hdots m$.  There will only be finitely many $V_i \cap \mathscr{G}_{U_{\alpha}}$ appearing in each of the finitely many $\mathscr{C}_i$ and the $\mathscr{G}_{U_{\alpha}}$ are disjoint, so without loss of generality we can rewrite
\begin{equation} \label{eq:0007}
f = \sum_{j=1}^n \sum_{i=1}^m  r_i \chi_{V_i \cap  \mathscr{G}_{U_{\alpha_j}}}
\end{equation}
Observe that $\chi_{V_i}{\restriction}_{\mathscr{G}_{U_{\alpha_j}}} = \chi_{V_i \cap  \mathscr{G}_{U_{\alpha_j}}}$ and therefore 
\begin{equation} \label{eq:0003}
f {\restriction}_{\mathscr{G}_{U_{\alpha_j}}} =\sum_{i = 1}^m  r_i \chi_{V_i}{\restriction}_{\mathscr{G}_{U_{\alpha_j}}} = \sum_{i = 1}^m  r_i \chi_{V_i \cap  \mathscr{G}_{U_{\alpha_j}}}    
\end{equation}
Since $\{ U_\alpha \}_{\alpha \in J}$ partition  $\mathscr{G}^{(0)}$ and $\{ \mathscr{G}_{U_\alpha} \}_{\alpha \in J}$ partition $\mathscr{G}$, then the following sum decomposition, where only finitely many of the summands are non-zero, for an arbitrary $f \in R \mathscr{G}$ is unique
\begin{equation} \label{eq:0004}
    f = \sum_{\alpha \in J}  f {\restriction}_{\mathscr{G}_{U_{\alpha}}}
\end{equation}
For arbitrary $f \in \mathscr{G}$, the support of $f$ is 
$$
\supp (f) = \{ \gamma \in \mathscr{G} | f(\gamma) \ne 0 \}.
$$
Then for all $f, g \in R \mathscr{G}$, we have $\supp(fg) \subseteq \supp(f) \supp(g)$, $\supp(f {\restriction}_{\mathscr{G}_{U_\alpha}}) \subseteq \mathscr{G}_{U_\alpha}$, and $\supp(g {\restriction}_{\mathscr{G}_{U_\beta}}) \subseteq \mathscr{G}_{U_\beta}$.  If $\alpha \ne \beta$, then $\supp(f {\restriction}_{\mathscr{G}_{U_\alpha}} \ast g {\restriction}_{\mathscr{G}_{U_\beta}}) \subseteq \mathscr{G}_{U_\alpha} \mathscr{G}_{U_\beta} = \emptyset$ and 
\begin{equation} \label{eq:0008}
f {\restriction}_{\mathscr{G}_{U_\alpha}} \ast g {\restriction}_{\mathscr{G}_{U_\beta}} = 0
\end{equation}
So $R \mathscr{G} = \bigoplus _{\alpha \in J} R \mathscr{G}_{U_{\alpha}}$ is a ring direct sum decomposition.
\end{proof}
\begin{Lemma} \label{thm:0004} Suppose $R$ is a commutative ring with unit and $\mathscr{G}$ is an ample groupoid.  If $\mathscr{G}$ is discrete, then $R \mathscr{G} = \bigoplus _{\alpha \in J} R \mathscr{G}_{\mathscr{O}_{\alpha}} \cong \bigoplus _{\alpha \in J}     M_{\mathscr{O}_{\alpha}}(R \mathscr{G}_{x_\alpha}^{x_\alpha})$ where $J$ is an index set for the orbits of $\mathscr{G}$, $\mathscr{O}_{\alpha}$ is the $\alpha$- \textit{th} orbit, $x_\alpha$ is a representative unit of $\mathscr{O}_\alpha$, and $\mathscr{G}_{x_\alpha}^{x_\alpha}$ is the isotropy group of $x_\alpha$.
\end{Lemma}
\begin{proof}
If $\mathscr{G}$ is discrete, then by  \cite[Proposition 5.3]{GCRandCCRJournal}, for a single orbit $\mathscr{O}$ in $\mathscr{G}$, $R \mathscr{G}_{\mathscr{O}} \cong M_{\mathscr{O}}(R \mathscr{G}_x^x) $, where $x$ is a representative unit of the orbit $\mathscr{O}$.  (Proposition 5.3 \cite{GCRandCCRJournal} is stated for the case where $R$ is a field, but the proof also holds in the more general setting where $R$ is a commutative ring with unit.)  The orbits are disjoint and form a partition of the unit space and since $\mathscr{G}$ is discrete the orbits are disjoint clopen invariant subsets of $\mathscr{G}^{(0)}$.  Therefore, the rest follows from Lemma \ref{thm:0003}.
\end{proof}
\section{Preliminaries on Chain Conditions}
\subsection{Classical Chain Conditions}
Recall the classical notions of ascending and descending chain conditions on an object such as a group, ring, or module.  An object satisfies the ascending (descending) chain condition if every ascending (descending) chain of subobjects has finite length.  When the object satisfies the ascending (descending) chain condition for its subobjects we say that the object is noetherian (artinian).
\newline 
When our object is a ring, the subobjects are left-  or right- ideals depending on whether we are investigating left- or right- noetherian (artinian).  In the case of groups the subobjects are subgroups.
\newline
We will need the following results.
\begin{Thm} [Hilbert's Basis Theorem 1.9 of \cite{GoodearlWarfield-NonCommNoether}] \label{thm:2011} Let $R[x]$ be a polynomial ring in one indeterminate.  If the coefficient ring $R$ is a left (right) noetherian unital ring then so is the polynomial ring R[x].
\end{Thm}
\begin{Thm}[Theorem 1 of \cite{connell}] \label{thm:0010} Let $RG$ be a group ring where $R$ is a unital ring and $G$ is a group.  $RG$ is left (right) artinian if and only if $R$ is left (right)  artinian and $G$ is finite.
\end{Thm}
\begin{Thm}[Theorem 2 of \cite{connell}] \label{thm:0011} Let $RG$ be a group ring where $R$ is a unital ring and $G$ is a group.  
If $RG$ is left (right) noetherian, then $R$ is left (right) noetherian and $G$ is noetherian.
\end{Thm}
\begin{Thm} [Theorem 2.7, Chapter 10 of \cite{PassmanBook}] \label{thm:0012}
Let $S$ be a unital ring, $R$ a left (right) noetherian subring with the same multiplicative identity, and $G$ a polycyclic by finite group of units in $S$.  If $Rg = gR$ for all $g \in G$ and $S = \langle R, G \rangle$, then $S$ is left (right) noetherian.
\end{Thm}
This gives us as a corollary a well-known result due to Hall.
\begin{Cor} [Theorem 1 of \cite{HallFinitenessConditions}] \label{thm:0013} Let $RG$ be a group ring where $R$ is a left (right) noetherian unital ring and $G$ is a polycyclic by finite group, then $RG$ is left (right) noetherian.
\end{Cor}
\begin{Rmk}
We observe that Hilbert's Basis Theorem \ref{thm:2011}, Connell's Results (Theorem \ref{thm:0010} and \ref{thm:0011}), Theorem \ref{thm:0012}, and Hall's Theorem (Corollary \ref{thm:0013}) do not require the base ring $R$ to be commutative and there are technical considerations to consider in the case where the base ring is noncommutative.  However, we will only apply these results in the instance where the coefficient ring is commutative. 
\end{Rmk}
To date the only known results relating chain conditions for group rings with conditions on the underlying coefficient ring and group are  Connell's results \cite{connell}, some of which are used in our results, and Hall \cite{HallFinitenessConditions}.  As we have already seen, Connell showed the group ring is artinian if and only if both the base ring is artinian and the group is finite.  The characterization of noetherian group rings is more involved.  Thus far the largest class of groups with noetherian group rings over a  noetherian coefficient ring are polycyclic-by-finite groups \cite{HallFinitenessConditions}, \cite[Chapter 10]{PassmanBook}.
\newline \newline
We will need the following corollary to the Hopkins-Levitzki Theorem.  Note that when discussing noetherian and artinian rings we assume that the rings are unital.
\begin{Thm} [Hopkins-Levitzki, Theorem 15.11 of \cite{RowenGrdAlgNonComm}] \label{thm:2019}
Any unital left (right) artinian ring $R$ is also left (right) noetherian
\end{Thm}
\subsection{Generalized Chain Conditions}
\begin{Def}[Local Units] \label{def:0001} A \textit{set of local units} for a ring $R$ is a set $E \subseteq R$ of idempotents with the property that for any finite subset $x_1, \hdots , x_n \in R$ there exists $e \in E$ such that $ex_i=x_i e=x_i$, $i = 1, \hdots n$.  We say that $R$ is a \textit{ring with local units} if $R$ contains a set of local units.  Equivalently, $R$ contains a set of local units if $R$ is a directed union of unital subrings.
\end{Def}
We remind the reader that Proposition \ref{prop:1004}(\ref{prop:1004_4}) implies that etale groupoid algebras have local units.
\begin{Def}[Enough Idempotents] \label{def:0002}  We say that a ring $R$ has \textit{enough idempotents} if there exists a collection of mutually orthogonal idempotents $\{ e_{\alpha} \}_{\alpha \in A}$ such that
$$
R = \bigoplus_{\alpha \in A} e_{\alpha}R = \bigoplus_{\alpha \in A} Re_{\alpha}
$$
where the above are module direct sum decompositions.
\end{Def}
\begin{Rmk}  \label{rmk:0002} If we let $S = \{ e_{\alpha} \}_{\alpha \in \Lambda}$ be mutually orthogonal idempotents of the above definition then $E := \{ \sum_{k=1}^n e_k:  e_1, \dots e_n \in S \}$ is a set of local units for $R$ \cite{tomfordeuniqueidealLPA}.  Thus rings with enough idempotents and groupoid algebras are rings with local units, where for groupoid algebras the set of idempotents $E$ is the set of characteristic functions of compact open sets in $\mathscr{G}^{(0)}$ (Proposition \ref{prop:1004}(\ref{prop:1004_4})).
\end{Rmk}
\begin{Def}\label{def:0009}
If $R$ is a ring, then by a left $R$-module we mean a unitary left $R$-module, that is, a module $M$ with the added condition that $RM = M$.  For a unital ring $R$, this is equivalent to $1 \cdot m = m$ for all $m \in M$ .
\end{Def}
Throughout this paper we will assume our modules are unitary.
\begin{Def}[Categorical Chain Conditions] \label{def:0006}   Let $R$ be a ring with local units.
\begin{enumerate}
    \item $R$ is \textit{categorically left (resp. right) artinian} if every finitely generated left (resp. right) $R$-module is artinian.
    \item $R$ is \textit{categorically left (resp. right) noetherian} if every finitely generated left (resp. right) $R$-module is noetherian.
\end{enumerate}
\end{Def}
\begin{Def}[Local Chain Conditions] \label{def:0007} Let $R$ be a ring with local units.
\begin{enumerate}
    \item $R$ is \textit{locally left (resp. right) artinian} if $eRe$ is left (resp. right) artinian for every idempotent $e \in R$.
    \item $R$ is \textit{locally left (resp. right) noetherian} if $eRe$ is left (resp. right) noetherian for every idempotent $e \in R$.
\end{enumerate}
\end{Def}
\begin{Rmk} \label{rmk:0008}
The classical chain conditions are quite restrictive.  For example, when the ring has local units (such as a groupoid algebra) and left (right) noetherian, then the ring must be unital.  For this reason the categorical and local chain conditions described above were developed.   A ring isomorphic to its opposite ring (e.g., rings with involution such as groupoid algebras) is left categorically (locally) noetherian (artinian) if and only if it is right categorically (locally) noetherian (artinian).  So when it is unnecessary to make a distinction between left and right categorical or local chain conditions, such as the main theorem (Theorem \ref{thm:1001}), we will omit the adjective “left” or “right” in describing the chain condition.
\end{Rmk}
We also observe that unital left (right) artinian rings are always left (right)  noetherian (Hopkins-Levitzki Theorem \ref{thm:2019}) so locally left (right) artinian implies locally left (right)  noetherian.
\section{Results on Generalized Chain Conditions}
We will need the following result of Abrams, Arana Pino, Perera and Siles Molina \cite{AbramsetalchaincondLPA}.
\begin{Prop} \label{prop1}
Suppose $R$ is a ring with local units and $E$ is a set of idempotents such that $R = \oplus_{e \in E} Re$.  Then $R$ is categorically left artinian (resp. noetherian) if and only if each $Re$ is a left artinian (resp. noetherian) $R$-module.  In particular, if $R$ is a unital ring, then $R$ is left artinian (resp. noetherian) if and only if $R$ is categorically left artinian (resp. noetherian).
\end{Prop}
\begin{Rmk} \label{rmk:1001}
The above proposition holds for categorically right artinian (resp. noetherian) when we replace $Re$ with $eR$ viewed as a right $R$-module, and each $eR$ must be right artinian (noetherian).  
\end{Rmk}
The following proposition is a well known result and will be called upon later.  
\begin{Prop} \label{prop:1002}
Let $f: R \rightarrow S$ be a homomorphism of rings.  If $M$ is an left $S$-module, then $M$ is a left $R$-module where the action of $R$ on $M$ is given by $r \cdot m = f(r) \cdot m$.  Furthermore, if $f$ is surjective, then $R$-submodules of $M$ are the same as $S$-submodules of $M$.
\end{Prop}
\begin{Rmk} \label{rmk:1002}
The analog of Proposition \ref{prop:1002} also holds for right $R$- and $S$-modules.
\end{Rmk}
\begin{Thm} \label{thm:1008}  Suppose $R_\alpha$ has enough idempotents for all $\alpha \in A$.  If $R_\alpha$ is categorically left (right) noetherian/artinian for all $\alpha \in A$ then $R = \bigoplus_{\alpha \in A} R_{\alpha}$ is categorically left (right) noetherian/artinian.
\end{Thm}
\begin{proof}
Let
\begin{equation} \label{eq:1046}
    R = \bigoplus_{\alpha \in A} R_{\alpha}
\end{equation}
where each $R_{\alpha}$ has enough idempotents and is categorically left (right) noetherian (artinian) $\forall \alpha \in A$.  Then for arbitrary $\alpha \in A$ there exists a collection of orthogonal idempotents $\{ e_{\alpha, i} \}_{i \in I_{\alpha}} \subseteq R_{\alpha}$ such that
\begin{equation} \label{eq:1047}
    R_{\alpha} = \bigoplus_{i \in I_{\alpha}} e_{\alpha,i}R_{\alpha} = \bigoplus_{i \in I_{\alpha}} R_{\alpha}e_{\alpha,i}
\end{equation}
So
\begin{equation} \label{eq:1048}
    R = \bigoplus_{\alpha \in A} R_{\alpha} = \bigoplus_{\alpha \in A} \bigoplus_{i \in I_{\alpha}} e_{\alpha,i}R_{\alpha} = \bigoplus_{\alpha \in A} \bigoplus_{i \in I_{\alpha}} R_{\alpha}e_{\alpha,i}
\end{equation}
\newline
By construction, for every $r,s\in R$ there exist unique decompositions $r=\sum_{\alpha \in A} r_{\alpha}$, $s=\sum_{\alpha \in A}s_{\alpha}$, where each $r_{\alpha}, s_{\alpha} \in R_{\alpha}$ and only finitely many of the $r_{\alpha}$'s and $s_{\alpha}$'s are non-zero and 
\begin{equation} \label{eq:1049}
    r_{\alpha}s_{\beta}= 
\begin{cases}
    0,& \text{if } \alpha \ne \beta \\
    r_{\alpha}s_{\alpha},& \text{otherwise where } r_{\alpha}s_{\alpha} \\  & \text{ is the product in } R_{\alpha}
\end{cases}
\end{equation}
Therefore, 
\begin{equation} \label{eq:1050}
rs = (\sum_{\alpha \in A}r_{\alpha})(\sum_{\beta \in A} s_{\beta})= \sum_{\alpha \in A} \sum_{\beta \in A} r_{\alpha}s_{\beta}=\sum_{\alpha \in A} r_{\alpha}s_{\alpha}  
\end{equation}
Since the $R_\alpha$'s are both subrings and 2-sided ideals of $R$,  \eqref{eq:1046} holds as both an internal direct sum of rings and internal direct sum of 2-sided ideals and therefore an internal bimodule direct sum, but \eqref{eq:1047} and \eqref{eq:1048} are only module internal direct sums.
For every $\alpha \in A$, we also have the well-known and standard projection, 
$$
\pi_{\alpha} : \bigoplus_{\beta \in A} R_\beta \twoheadrightarrow R_{\alpha} 
$$
$$
\sum_{\beta \in A}r_\beta \xmapsto{\pi_{\alpha}} r_{\alpha}
$$
where $r_\beta \in R_\beta$, and inclusion maps
$$
\iota_{\alpha} :  R_{\alpha} \hookrightarrow  \bigoplus_{\beta \in A} R_\beta
$$
$$
r_{\alpha} \xmapsto{\iota_{\alpha}} r_\alpha 
$$
Here $\pi_\alpha$ and $\iota_\alpha$ are surjective and injective ring homomorphisms, respectively.
\newline \newline
From Proposition \ref{prop:1002} we can make any $R_\alpha$-module $M$ into an $R$-module via $\pi_\alpha$.  In particular, going forward this is the assumed action of $R$ on the left $R_\alpha$-module $R_\alpha e_{\alpha, i}$.
\newline \newline
It is clear from \eqref{eq:1046}, \eqref{eq:1049} and \eqref{eq:1050} that $\{ e_{\alpha,i} \}_{\alpha \in A, i \in I_\alpha}$ is a set of orthogonal idempotents in $R$. 
\newline \newline
For an arbitrary $\alpha \in A$, $r \in R$, and idempotent $e \in R_\alpha$,
$$
r = \sum_{\beta \in A}  r_\beta
$$
where all but finitely many $r_\beta$ are zero and $r_\beta \in R_\beta$.  Then $re = r_\alpha e \in R_\alpha e$.  So $Re \subseteq R_\alpha e$.  And since $R_\alpha \subseteq R$, $R_\alpha e \subseteq Re$.  Therefore,
\begin{equation} \label{eq:1063}
R e = R_\alpha e
\end{equation}
as abelian groups.  For arbitrary $r \in R$ and $s \in R_\alpha e$ the action is given by
\begin{equation} \label{eq:1065}
    r s = r_\alpha s = \pi_\alpha(r) s
\end{equation}
So the left action of $R$ on $Re=R_{\alpha}e$ by left multiplication gives the same module structure as the module structure on $R_{\alpha}e$ coming from the projection $\pi_{\alpha}:R \rightarrow R_{\alpha}$ via Prop \ref{prop:1002}.  So the submodules of $Re$ as an $R$-module are the same as the submodules of $R_\alpha e$ as an $R_\alpha$-module.  In particular, $Re$ is a left noetherian (artinian) $R$-module if and only if $R_\alpha e$ is a left noetherian (artinian) $R_\alpha$-module.
\newline
By symmetric arguments for an idempotent $e \in R_\alpha$
\begin{equation} \label{eq:1064}
eR  = e R_\alpha   
\end{equation}
and $eR$ is a right noetherian (artinian) $R$-module if and only if $e R_\alpha$ is a right noetherian (artinian) $R_\alpha$-module.
\newline \newline
Now we see from \eqref{eq:1046}, \eqref{eq:1047}, \eqref{eq:1048}, \eqref{eq:1063} and \eqref{eq:1064} that  
\begin{equation} \label{eq:1062} 
R = \bigoplus_{\alpha \in A} R_\alpha = \bigoplus_{\alpha \in A} \bigoplus_{i \in I_{\alpha}} R e_{\alpha, i} = \bigoplus_{\alpha \in A} \bigoplus_{i \in I_{\alpha}} e_{\alpha, i} R
\end{equation}
and by Proposition \ref{prop1}, $R$ is categorically left (right) noetherian (artinian).
\end{proof}
\begin{Rmk} \label{rmk:1003}
    As we will continue to make use of the following, we point out to the reader that \eqref{eq:1049} and \eqref{eq:1050} do not depend on the condition of the ring having enough idempotents.  This is, in fact, the rule for multiplication between ring elements when the ring has a bimodule decomposition into 2-sided ideals.  
\end{Rmk}
For an arbitrary ring $S$, not necessarily unital, we define $M_J(S)$ to be the ring of $J \times J$ matrices with entries in $S$ and only finitely many non-zero entries and $C_{J,p}(S)$ ($R_{J,p}(S)$) to be the set of $J \times J$ matrices with entries from $S$, only finitely many non-zero entries in column (row) $p$ and zero in all other columns (rows).
\newline \newline
When $S$ is a unital ring, we have the standard matrix units $ \{ E_{i,j} \}_{i,j, \in J}$ with $1_S$ in entry $i,j$ and $0$ elsewhere,
\begin{equation} \label{eq:1067}
\begin{aligned}
C_{J,p}(S) ={} & \{ M E_{p,p}: M \in M_J(S) \} = M_J(S) E_{p,p} \\
\end{aligned}
\end{equation}
and 
\begin{equation} \label{eq:1068}
\begin{aligned}
R_{J,p}(S) ={} & \{ E_{p,p}M : M \in M_J(S) \} = E_{p,p} M_J(S) \\
\end{aligned}
\end{equation}
\begin{Lemma} \label{thm:1009}
Let $S$ be a unital ring, when $C_{J,p}(S)$ is viewed as a left $M_J(S)$-module, then any left $M_J(S)$-submodule $N \subseteq C_{J,p}(S)$ has the form $C_{J,p}(\mathfrak{A})$ for a uniquely determined left ideal $\mathfrak{A}$ of $S$.  In particular, there is an inclusion-preserving bijection between the left $M_J(S)$-submodules of $C_{J,p}(S)$ and the left ideals of $S$.  The dual result  holds for $R_{J,p}(S)$.
\end{Lemma}
\begin{proof}
This proof is similar to Theorem 3.1 in \cite{LamBook}.  Clearly, if $\mathfrak{A}$ is a left ideal of $S$, then $C_{J,p}(\mathfrak{A})$ is a left $M_J(S)$-module.  Whenever $\mathfrak{A}$ and $\mathfrak{B}$ are left ideals of $R$, it is clear that $\mathfrak{A} \subseteq \mathfrak{B}$ if and only if $C_{J,p}(\mathfrak{A}) \subseteq C_{J,p}(\mathfrak{B})$.
\newline \newline
Let $N \le C_{J,p}(S)$ = $M_J(S)E_{p,p}$ be an arbitrary submodule of $C_{J,p}(S)$  and let $\mathfrak{A}$ be the set of all $(p,p)$-entries of matrices in $N$.  The $\{ E_{i,j} \}_{i,j \in J}$ form a basis for $M_J(S)$ as a left $S$-module.  $\mathfrak{A}$ is easily seen to be a left ideal of $S$.  For arbitrary $x_1, x_2 \in \mathfrak{A}$ and arbitrary $r \in S$, $x_1, x_2$ must be the $(p,p)$-entries of matrices $M_1, M_2 \in N$, respectively and $r E_{p,p} M_1 - M_2$ is also a matrix in the submodule $N$  and the $(p,p)$-entry is $rx_1-x_2$.  So we can easily see that $\mathfrak{A}$ is a left ideal of $S$.  All that remains is for us to show that $N = C_{J,p}(\mathfrak{A})$.
\newline \newline
For any matrix $M = (m_{i,j})$ we have the following identities
\begin{equation} \label{eq:1}
    ME_{k,l}= 
\begin{cases}
    \text{matrix with 0 in all entries except column $l$} \\
    \text{and column $l$ equals column $k$ from matrix } M \\
\end{cases}
\end{equation}
\begin{equation} \label{eq:2}
    E_{i,j}M= 
\begin{cases}
    \text{matrix with 0 in all entries except row $i$} \\
    \text{and row $i$ equals row $j$ from matrix } M \\
\end{cases}
\end{equation}
\begin{equation} \label{eq:3}
    E_{i,j}ME_{k,l} = m_{j,k}E_{i,l}
\end{equation}
where $\{ E_{i,j} \}_{i,j \in J}$ are the matrix units.  
For any $M \in N$, we know that $M = ME_{p,p}$.  So if we let $i = p$ in \eqref{eq:2} and \eqref{eq:3} and implicitly $k = l = p$ in \eqref{eq:3}, we see that $E_{p,j}M = E_{p,j}ME_{p,p}= m_{j,p}E_{p,p} \in N$  which implies $m_{j,p} \in \mathfrak{A}$ for all $j \in J$.  Also it is obvious that $m_{j,k} = 0 \in \mathfrak{A}$ whenever $k \ne p$.  So $N \subseteq C_{J,p}(\mathfrak{A})$.  
\newline \newline
We now prove $C_{J,p}(\mathfrak{A}) \subseteq N$.  Take any $(a_{i,j}) \in C_{J,p}(\mathfrak{A})$.  So $a_{i,j}=0$ whenever $j \ne p$, it is sufficient for us to show that $a_{i,p}E_{i,p} \in N$ for any $i \in J$.  By construction of $\mathfrak{A}$ there exists a matrix $(m_{i,j})=M \in N \subseteq C_{J,p}(\mathfrak{A})$ such that $a_{i,p}=m_{p,p}$.  Then we have $m_{p,p} \in \mathfrak{A}$ and 
$$
a_{i,p}E_{i,p}= m_{p,p}E_{i,p}=E_{i,p}M \in N
$$
Thus we have our desired result, including an explicit inclusion-preserving bijection between left $M_J(S)$-submodules of $C_{J,p}(S)$ and left ideals of $S$. 
\newline \newline
The proof for right $M_J(S)$-submodules of $R_{J,p}(S)$ and right ideals of $S$ is dual.
\end{proof}
\begin{Thm} \label{thm:1005}
Let $S$ be a unital ring and $M_J(S)$ denote $J \times J$ matrices with entries in S and only finitely many non-zero entries.  If $S$ is left (respectively right) noetherian/artinian if and only if $M_J(S)$ is categorically left (respectively right) noetherian/artinian.
\end{Thm}
\begin{proof}
($\Rightarrow$) Lemma \ref{thm:1009} implies that $C_{J,p}(S)$  $(R_{J,p}(S))$ is left (right) noetherian/artinian if and only if $S$ is left (right) noetherian/artinian.
We observe that $\{ E_{p,p} \}_{p \in J}$ form a collection of orthogonal idempotents in $M_J(S)$,
\begin{equation}\label{eq:1069}
    \bigoplus_{p \in J} C_{J,p}(S) = \bigoplus_{p \in J}M_J(S)E_{p,p} = M_J(S) =   \bigoplus_{p \in J}E_{p,p}M_J(S) = \bigoplus_{p \in J} R_{J,p}(S)
\end{equation}
i.e., $\{ E_{p,p} \}_{p \in J}$ form a collection of enough idempotents in $M_J(S)$.
\newline
($\Leftarrow$) $C_{J,p}(S) = M_J(S)E_{p,p}$  $(R_{J,p}(S) = E_{p,p}M_J(S))$ is a finitely generated left (right) $M_J(S)$-module and hence left (right) noetherian/artinian.  Again by Lemma \ref{thm:1009}, $(R_{J,p}(S))$ is left (right) noetherian/artinian if and only if $S$ is left (right) noetherian/artinian.
The result now follows from Proposition \ref{prop1}.
\end{proof}
\begin{Thm} \label{thm:1010}
Let $S_\alpha$ be unital left (right) noetherian (artinian) rings for all $\alpha \in A$, then $\bigoplus_{\alpha \in A} M_{J_\alpha}(S_\alpha)$ is categorically left (right) noetherian (artinian).
\end{Thm}
\begin{proof}
Follows immediately from Theorems \ref{thm:1008} and \ref{thm:1005}.
\end{proof}
The following result is stated as Lemmas 1.5 and 1.6 in \cite{AbramsetalchaincondLPA}; however, a proof is only given for the  categorically left (right) artinian implies locally left (right) artinian  case.  We state the result here with proofs for both the artinian and noetherian settings.
\begin{Thm} \label{thm:1011}
Let $S$ be a ring with local units.  If $S$ is categorically left (right) noetherian/artinian then $S$ is locally left (right) noetherian/artinian.
\end{Thm}
\begin{proof}
We will prove the \textit{left} version of this theorem, the proof of the \textit{right} version is dual.  The result follows from Theorem 21.11 of \cite{LamBook} which states that for any ring $S$ there is an inclusion preserving injection $\phi$ from left $eSe$-submodules of $eSe$ into left $S$-submodules of $Se$ (This theorem is stated for arbitrary rings with unit; however, the proof of this result does not require the ring to have a unit and therefore, applies to our setting.)  The ring $S$ must be locally left noetherian/artinian.  Otherwise for some corner $eSe$ there will be a strictly ascending/descending chain of submodules and $\phi$ would take this chain of submodules to a strict chain of $S$-submodules of $Se$, where $Se$ is a finitely generated unitary $S$-module.  This is impossible, since by assumption our ring is categorically left noetherian/artinian, and any ascending/descending chain of $S$-submodules of $Se$ must stabilize.
\end{proof}
\subsection{Semisimplicity and Local Semisimplicity Results}
\begin{Def} [Minimal and simple modules/ideals]
Let $M$ be a non-zero left (right) $R$-module.  We say that $M$ is \textit{minimal} left (right) module if it contains no proper non-zero left (right) submodules.  A \textit{simple left (right) module} is a minimal left (right) module that is not annihlated by $R$.  A left (right) ideal $I$ of $R$ is minimal/simple if it is minimal/simple as an $R$-module.
\end{Def}
Equivalently, we can define simple left (right) modules to be minimal left (right) modules which are unitary.  Consequently, semisimple modules must also be unitary.  The notion of a \textit{left (right) semisimple} module will be familiar to the reader as a module that is the direct sum of simple modules and a \textit{left (right) semsimple ring} is a ring with local units that is semisimple as a left (right) module over itself.  Frequently, we make use of the equivalence of sums of simple modules and direct sums of simple submodules in proving semisimplicity.  The statement of this equivalence is found in the standard literature for unital rings, but the proof does not require the ring to be unital.  We further note that for rings with local units, all left/right ideals are unitary; and, in particular, all minimal left/right ideals are simple.
\begin{Rmk} \label{rmk:2007}
It is known to experts that left and right semisimple are equivalent even in the non-unital case, thus we omit the adjectives left and right when referring to semsimplicity.
\end{Rmk}
\begin{Def} \label{def:0011}  [Locally semisimple] \label{def:0010}  A ring $R$ with local units is \textit{locally semisimple} if for every idempotent $e \in R$, $eRe$ is a semisimple ring.
\end{Def}
\begin{Prop} \label{prop:1005}  
Locally semisimple implies locally artinian.
\end{Prop}  
\begin{Prop} \label{thm:1016}
Let $A$ be an arbitrary index set and $S_\alpha$ a ring with local units for all $\alpha \in A$, then $S = \oplus_{\alpha \in A} S_\alpha$ is a ring with local units.
\end{Prop}
\begin{proof}
    Since each $S_\alpha$ has local units, then $S$ has local units.  
    Let 
    $$
        s = \sum_{\alpha \in A} s_{\alpha} \text{ and } t = \sum_{\alpha \in A} t_{\alpha}         
    $$
    The following algorithm describes our construction of a local unit for both $s$ and $t$.  The construction is a finite process over the indices $\alpha \in A$ with non-zero values in either $s_{\alpha}$ or $t_{\alpha}$.  Let $\alpha \in I \subseteq A$ be the indices for which either $s_{\alpha} \ne 0$ or $t_{\alpha} \ne 0$. Then for $\alpha \in I$, we know that there is an idempotent $e_{\alpha} \in S_{\alpha}$ such that 
    $$
    s_{\alpha} e_{\alpha} = s_{\alpha} = e_{\alpha} s_{\alpha} 
    $$
    \text{ and } 
    $$t_{\alpha} e_{\alpha} = t_{\alpha} = e_{\alpha} t_{\alpha}.$$
Let $e = \sum_{\alpha \in I} e_{\alpha}$, then it is clear (from  \eqref{eq:1049} / \eqref{eq:1050} and Remark \ref{rmk:1003}) that
$$
se = s = es \text{ and } te = t = et.
$$
\end{proof}
\begin{Prop} \label{prop:1009}  Suppose $R$ is a ring (not necessarily unital).  If $M$ is a finitely generated semisimple left (right) $R$-module, then $End_R(M)$ is a semisimple unital ring. 
\end{Prop}  
\begin{proof}
Since $M$ is finitely generated and semisimple, $M = \bigoplus_{i=1}^k  M_i^{n_i}$, where the $M_i$’s are non-isomorphic simple left $R$-modules.  Then \begin{equation} \label{eq:2015}
End_R(M) \cong [Hom_R(M_i^{n_i}, M_j^{n_j})] =  \prod_{i=1}^k  M_{n_i}(D_i) 
\end{equation}
where $f \in End_R(M)$ is mapped to the matrix $[f_{i,j}]$, $f_{i,j} \in Hom_R(M_i^{n_i}, M_j^{n_j})$ acting on $M$ from the right.   By Schur's Lemma, $D_i \cong End_R(M_i)$ is a division ring and $Hom_R(M_i^{n_i}, M_j^{n_j})=\{0 \}$ for $i \ne j$.  We see that $End_R(M)$ is semisimple by the Wedderburn-Artin theorem.  The proof for the right semisimple hypothesis requires analogous modifications to the preceding.
\end{proof}
\begin{Prop} \label{prop:1006a}
If $S$ is a ring which is semisimple as a left (right) $S$-module, then every corner is semsisimple.  In particular, if $S$ is a semisimple ring, then $S$ is locally semisimple.
\end{Prop}
\begin{proof}
We will prove the statement for the left semisimple hypothesis, the analogous statement can be with proved with syntactical changes for right semisimplicity.  For any ring $S$ and idempotent $e \in S$, $Se$ is a left submodule of $S$ over $S$.  Since $S$ is left semisimple, so is $Se$ and $End_S(Se) \cong eSe$.  Since $Se$ is a finitely generated left $S$-module and left semisimple as an $S$-module and along with Proposition \ref{prop:1009}, we conclude that $End_S(Se) \cong eSe$ [Proposition 21.6 \cite{LamBook}].
\end{proof}
\begin{Rmk} \label{rmk:2006}
    It is known that the converse of Proposition \ref{prop:1006a} holds, but we will not prove this here.
\end{Rmk}
\begin{Prop} \label{thm:1015}  Let $S = \oplus_{\alpha \in J} S_\alpha$ where each $S_\alpha$ is a ring not necessarily unital and $I \subseteq S_\beta$ for some $\beta \in J$.  $I$ is a left (right) ideal of $S$ if and only if $I$ is a left (right) ideal of $S_\beta$.  In particular, $I$ is a minimal left (right) ideal of $S$ if and only if $I$ is a minimal left (right) ideal of $S_{\beta}$.
\end{Prop}
\begin{proof}
For arbitrary $\beta \in J$, let $I \subseteq S_\beta$. If $I$ is an ideal of any ring it is an abelian group.  So in both directions we need only check that $I$ is closed under the multiplicative action of $S$ or $S_\beta$ where appropriate.  \newline
($\Leftarrow$)  Let $I$ be a left (right) ideal  of $S_\beta$.  For an arbitrary $m \in I \unlhd S_\beta$ and $s \in S$, we have the following unique decomposition
\begin{equation} \label{eq:1071}
    s = \sum_{\alpha \in J} s_\alpha
\end{equation}
where $s_\alpha \in S_\alpha$ and all but finitely many of the $s_\alpha$ are zero.  Since $m \in I \unlhd S_{\beta}$ by the same reasoning from Equation \eqref{eq:1049}, $sm = s_{\beta} m \in I $. \newline
($\Rightarrow$)  This direction is obvious, since $S_\beta$ is a subring of $S$.  Therefore, the closure of $I$ under the multiplicative action of $S$ gives us the closure of $I$ under the multiplicative action of $S_\beta$.  
\end{proof}
\begin{Prop} \label{thm:1014}
Let $A$ be an arbitrary index set and $S_\alpha$ a semisimple ring for all $\alpha \in A$, then $S = \oplus_{\alpha \in A} S_\alpha$ is semisimple.
\end{Prop}
\begin{proof}
Since each $S_\alpha$ is semisimple, $S_\alpha = \sum_{i \in J_\alpha} L_{\alpha,i}$  where each $L_{\alpha,i}$ is a minimal left ideal of $S_\alpha$.  Since by Proposition \ref{thm:1015}, each $L_{\alpha, i}$ is also a minimal left ideal of $S$ we have $S = \sum_{\alpha \in J} \sum_{i \in J_\alpha} L_{\alpha,i}$.  Each $S_{\alpha}$ has local units, therefore, by Proposition \ref{thm:1016}, $S$ has local units, and is therefore, semisimple.
\end{proof}
The following propositions are well-known results for the case where the indices are finite, and we extend them to the case of arbitrary indices.  
\begin{Prop} \label{thm:1013}
               Let $S$ be a unital ring, $J$ an arbitrary index set, and $M_J(S)$ the set of all matrices with entries in $S$ with finitely many non-zero entries in every row and column.  Then $S$ is semisimple if and only if $M_J(S)$ is semisimple.
\end{Prop}
\begin{proof}
Almost completely follows from Lemma \ref{thm:1009}.  Since $S$ is a semisimple ring, 
\begin{equation} \label{eq:2018}
S = \bigoplus_{i \in A} I_i
\end{equation}
where each $I_i$ is a minimal left ideal of $S$.  The left $M_J(S)$-submodules of $C_{J,p}(S)$ and the left ideals of $S$ are in an order preserving bijection, so 
\begin{equation} \label {eq:2019}
    C_{J,p}(S) = \sum_{i \in A} C_{J,p}(I_i) 
\end{equation}
where each $I_{i}$ is a minimal left ideal of $S$, and therefore, each $C_{J,p}(I_i)$ is a minimal left $M_J(S)$-submodule of $C_{J,p}(S)$.  Since $S$ is a unital ring and $M_J(S)$ is a ring with local units, the minimal left ideals of $M_J(S)$ are also simple left ideals.
It is clear from \eqref{eq:1067}, \eqref{eq:1068}, \eqref{eq:1},  \eqref{eq:2}, and \eqref{eq:2019} that 
\begin{equation} \label{eq:2016}
M_J(S) = \bigoplus_{p \in J} C_{J,p}(S) = \bigoplus_{p \in J} \sum_{i \in A} C_{J,p}(I_i) 
\end{equation}
as a direct sum of left modules.  
\newline
($\Leftarrow$)  $S$ is isomorphic to a corner $E_{p,p}M_J(S)E_{p,p}$ for any $p \in J$.  The result follows from Proposition \ref{prop:1006a}.
\end{proof}
\section{Main Result}
\begin{Thm} \label{thm:1001} If $R$ is a commutative ring with unit and $\mathscr G$ is an ample groupoid then the following are equivalent:
\begin{enumerate}
  \item \label{thm:1001_01} $R \mathscr{G}$ is categorically noetherian (artinian).
  \item \label{thm:1001_02} $R \mathscr{G}$ is locally noetherian (artinian).
  \item \label{thm:1001_03} $\mathscr{G}$ is discrete and $R \mathscr{G}_x^x$ is noetherian (artinian) for all isotropy groups $\mathscr{G}_x^x$.
  \item \label{thm:1001_04} $R \mathscr{G} \cong \bigoplus _{\alpha \in J} M_{\mathscr{O}_{\alpha}}(R \mathscr{G} _{x_\alpha}^{x_\alpha})$ where $x_\alpha$, $\alpha \in J$, represent the orbits, $\mathscr{O}_{\alpha}$ is the orbit of $x_\alpha$, $\mathscr{G}_{x_\alpha}^{x_\alpha}$ are the isotropy groups at $x_\alpha$, and $R \mathscr{G}_{x_\alpha}^{x_\alpha}$ are noetherian (artinian). In the case where the $R \mathscr{G}_{x_\alpha}^{x_\alpha}$ are artinian, the isotropy groups are finite and $R$ is artinian.
\end{enumerate}
\end{Thm}
\begin{proof}
(1) implies (2):  Follows from Theorem \ref{thm:1011}.
\newline \newline
(2) implies (3):  Assume $R \mathscr{G}$ is locally noetherian.  Let $U \subseteq \mathscr{G}^{(0)}$ be an arbitrary basic compact open set, hence $\chi_U \in R \mathscr{G}$.  Let $\{ U_i \}_{i \in I}$ be a chain of compact open sets such that each $U_i \subseteq U$ where $U_1 \subseteq U_2 \subseteq U_3 \hdots \subseteq U_n \subseteq \hdots $.  Recall that if $A$ is a ring and $e, f \in A$ are idempotents, then $Ae \subseteq Af$ if and only if $ef = e$.  Then each $\chi_U R \mathscr{G} \chi_U \chi_{U_i} = \chi_U R \mathscr{G} \chi_{U_i}$ is a left ideal of $\chi_U R \mathscr{G} \chi_U$ and by assumption $\chi_U R \mathscr{G} \chi_U$ is noetherian.  Applying the above observation to $A = \chi_U R \mathscr{G} \chi_U$ and using Proposition \ref{thm:0009}, we obtain the ascending chaiin
\begin{equation} \label{eq:1011}
    \chi_U R \mathscr{G} \chi_{U_1} \subseteq \chi_U R \mathscr{G} \chi_{U_2} \subseteq \hdots \chi_U R \mathscr{G} \chi_{U_n} \subseteq \hdots
\end{equation}
that must stabilize since $R \mathscr{G}$ is locally noetherian.  We also see that the chain $U_1 \subseteq U_2 \subseteq U_3 \hdots \subseteq U_n \subseteq \hdots $ must stabilize otherwise \eqref{eq:1011} would be a strictly ascending chain in the noetherian ring $\chi_U R \mathscr{G} \chi _U$ by Proposition \ref{thm:0009}.  So $U$ is Hausdorff with a basis of compact open sets satisfying the ascending chain condition and by Lemma \ref{thm:0001}, $U$ must be finite and therefore discrete.  This gives us that $\mathscr{G}^{(0)}$ and hence $\mathscr{G}$ is discrete, as $\mathscr{G}^{(0)}$ is Hausdorff with a basis of finite sets.
\newline \newline
Since $\mathscr{G}$ is discrete, for each $x \in \mathscr{G}^{(0)}$, $x$ is an isolated point and $\{ x \}$ is a compact open set, so $\chi_{ \{ x \}}$ is an idempotent and therefore by Proposition \ref{prop:1004}(\ref{prop:1004_3}) $R \mathscr{G}_x^x = R \mathscr{G}{\restriction_{ \{ x \} }} = \chi_{ \{ x \}} R \mathscr{G} \chi_{ \{ x \}}$ is noetherian, since $R \mathscr{G}$ is locally noetherian.
\newline \newline
As mentioned earlier locally artinian implies locally noetherian.  So the above proof shows $R \mathscr{G}$ locally artinian implies $\mathscr{G}$ is discrete and a similar argument shows that all of the isotropy group rings $R \mathscr{G}_{x_\alpha}^{x_\alpha}$ must be artinian.
\newline \newline  
(3) implies (4):  The direct sum decomposition $R \mathscr{G} \cong \bigoplus _{\alpha \in J} M_{\mathscr{O}_{\alpha}}(R \mathscr{G}_{x_\alpha}^{x_\alpha})$ is immediate from Lemma \ref{thm:0004}. The isotropy group rings $R \mathscr{G}_{x_\alpha}^{x_\alpha}$  are noetherian (artinian) by assumption.  A well-known result of Connell (Theorem \ref{thm:0010}) tells us that when the $R \mathscr{G}_{x_\alpha}^{x_\alpha}$ are artinian, the isotropy groups must be finite and $R$ must be artinian.
\newline \newline
(4) implies (1):  Since groupoid algebras have an involution, each $R \mathscr{G}_{x_\alpha}^{x_\alpha}$ is left noetherian (artinian) if and only if it is also right noetherian (artinian).  The result follows immediately from Theorem \ref{thm:1010}.  
\end{proof} 
\begin{Thm} \label{thm:1012} If $R$ is a commutative ring with unit and $\mathscr G$ is an ample groupoid then the following are equivalent:
\begin{enumerate}
  \item \label{thm:1012_01} $R \mathscr{G}$ is semisimple 
  \item \label{thm:1012_02} $R \mathscr{G}$ is locally semisimple.
  \item \label{thm:1012_03} $\mathscr{G}$ is discrete and $R \mathscr{G}_x^x$ is semisimple for all isotropy groups $\mathscr{G}_x^x$.
  \item \label{thm:1012_04} $R \mathscr{G} \cong \bigoplus _{\alpha \in J} M_{\mathscr{O}_{\alpha}}(R \mathscr{G} _{x_\alpha}^{x_\alpha})$ where $x_\alpha$, $\alpha \in J$, represent the orbits, $\mathscr{O}_{\alpha}$ is the orbit of $x_\alpha$, $\mathscr{G}_{x_\alpha}^{x_\alpha}$ are the isotropy groups at $x_\alpha$, and $R \mathscr{G}_{x_\alpha}^{x_\alpha}$ are semisimple. 
\end{enumerate}
\end{Thm}
\begin{proof}
(\ref{thm:1012_01}) implies (\ref{thm:1012_02}):  The result follows from Proposition \ref{prop:1006a}.
\newline \newline
(\ref{thm:1012_02}) implies (\ref{thm:1012_03}):  By assumption, every corner of $R \mathscr{G}$ is a unital semisimple ring and therefore artinian.  Hence $R \mathscr{G}$ is locally artinian and by Theorem \ref{thm:1001} $\mathscr{G}$ is discrete.  As in the proof of the Theorem \ref{thm:1001} \textit{(\ref{thm:1001_02}) implies (\ref{thm:1001_03})}, the discreteness of $\mathscr{G}$ and local semisimplicity of $R \mathscr{G}$ yields that for every $x \in \mathscr{G}^{(0)}$ $R \mathscr{G}_x^x = \chi_{\{ x \}} R \mathscr{G} \chi_{\{ x \}}$ is a corner in $R \mathscr{G}$ and hence semisimple.  
\newline \newline
(\ref{thm:1012_03}) implies (\ref{thm:1012_04}):  As in the proof of Theorem \ref{thm:1001}, the direct sum decomposition $R \mathscr{G} \cong \bigoplus _{\alpha \in J} M_{\mathscr{O}_{\alpha}}(R \mathscr{G}_{x_\alpha}^{x_\alpha})$ is immediate from Theorem \ref{thm:0004} and the isotropy groups are semisimple by assumption on $R \mathscr{G}_{x_\alpha}^{x_\alpha}$. 
\newline \newline
(\ref{thm:1012_04}) implies (\ref{thm:1012_01}):  Follows immediately from Propositions \ref{thm:1014}  and \ref{thm:1013}.
\end{proof}
\section{Applications}
\subsection{Leavitt Path Algebras}
For an arbitrary directed graph $E$ there is an $R$ algebra $L_R(E)$, called the Leavitt path algebra. Many ring-theoretic properties of $L_R(E)$ are controlled by the graphical properties of $E$. \cite{groupoidapproachleavitt}.  In \cite{LeavittBook}, the authors described categorical and local chain conditions for Leavitt path algebras of row finite graphs, conditions under which they coincide and when these algebras are semisimple.  \cite{LeavittBook} also relates these conditions to the graphical properties of $E$.  We will show that the conditions in \cite{LeavittBook} coincide with topological conditions on an ample groupoid derived from the graph and recover these results as special cases of our main result, Theorem \ref{thm:1001}.
\newline \newline
Let $E = (E^0,E^1, s, r)$ be a (directed) graph with vertex set $E^0$ and edge set $E^1$. We use 
$$
s:E^1 \rightarrow E^0
$$
$$
e \mapsto s(e)
$$
$s(e)$ for the source of an edge $e$ and 
$$
r:E^1 \rightarrow E^0
$$
$$
e \mapsto r(e)
$$
$r(e)$ for the range, or target, of an edge, $e$. A vertex $v$ is called a \textit{sink} if $s^{-1}(v) = \emptyset$ and it is called an \textit{infinite emitter} if $|s^{-1}(v)|= \infty$. The length of a finite (directed) path $\alpha$ is denoted $|\alpha|$ \cite{SteinbergGroupoidAlgebra}, \cite{groupoidapproachleavitt}.
\newline
A finite path is a finite sequence of edges $\alpha = \alpha_1 \alpha_2 \cdots \alpha_n$ such that $r(\alpha_i) = s(\alpha_{i+1})$ for all $i = 1, \cdots, n-1$. An infinite path is an infinite sequence of edges $\alpha = \alpha_1 \alpha_2 \cdots \alpha_n \alpha_{n+1} \cdots $ such that $r(\alpha_i) = s(\alpha_{i+1})$ for all $i$ in some total preordered index set $I$.   If $\alpha$ is a finite path, $r(\alpha) = r(\alpha_{|\alpha|})$ and if $\alpha$ is a finite or infinite path, $s(\alpha) = s(\alpha_1)$.  Let $\alpha$ be a finite path of positive length such that $r(\alpha) = s(\alpha) = v$ then we say that $\alpha$ is a closed path based at $v$.  A closed path $\alpha= \alpha_1 \cdots \alpha_{|\alpha|}$ with the property that none of the vertices $s(\alpha_1), \cdots, s(\alpha_{|\alpha|})$ are repeated is called a \textit{cycle}, and a graph without cycles is called \textit{acyclic}.  Let $I$ be a total preordered set used to index the edges in a path $\alpha$, an exit for a path $\alpha$ is an edge $f \in E^1$ with $s(f) = s(\alpha_i)$ for some  $i \in I$, and $f \ne \alpha_i$.  \cite{groupoidapproachleavitt}  Although we have extended the source $s$ and range maps $r$ to paths, $r^{-1}(v), s^{-1}(v) \subseteq E^1$ for all $v \in E^0$.
\begin{Def}
A \textit{ray} is an infinite path with no repeated vertex.
\end{Def}
Let $E^{\ast}$ be the set of all finite paths (including vertices) in $E$ and $E^{\infty}$ be the set of all infinite paths in $E$.  We define the path space of the graph $E$ as the set of all finite and infinite paths with an explicit basis of open sets for the topology\cite{groupoidapproachleavitt}.  We call $\partial E$ the \textit{boundary path of $E$}, where $\partial E$ is the set of all infinite paths and paths ending in a sink or an infinite emitter. 
\newline
The boundary path groupoid of the graph $E$ is defined as follows:
\begin{itemize}
    \item $\mathscr{G}_E = \{ (\alpha \gamma , |\alpha|-|\beta| , \beta \gamma ) \in \partial E \times \mathbb{Z} \times \partial E \mid |\alpha|, |\beta| < \infty \}$
    \item $\mathscr{G}_E^{(0)} =\{ (\gamma , 0 , \gamma ) \in \partial E \times \mathbb{Z} \times \partial E \mid |\alpha|, |\beta| < \infty \}$ and we identify $\mathscr{G}_E^{(0)}$ with $\partial E$ in the obvious way.
\end{itemize}
$d(\xi,k,\eta) = \eta, r(\xi,k,\eta) = \xi$ and $(\xi, k, \eta)(\eta, m, \zeta) = (\xi, k+m, \zeta)$.  Note that $\mathscr{G}_E$ is an ample groupoid with the topology given explicitly by a basis of compact open sets \cite{groupoidapproachleavitt}.  We will only need to work with the topology on $\mathscr{G}_E^{(0)}$ to show that it is discrete.  For this reason we will describe and work with the basis sets for $\partial E$ which corresponds to the topology on $\mathscr{G}^0_E$ under identification with $\partial E$.
\newline
Let $\alpha \in E^{\ast}$ be a finite path, we define the cylinder set
\begin{equation} \label{eq:2013}
    C(\alpha) = \{ \alpha x | x \in E^{\ast} \cup E^{\infty}, r(\alpha) = s(\alpha) \} \subseteq E^{\ast} \cup E^{\infty}
\end{equation}
\begin{Rmk} \label{rmk:2001}
Suppose $x$ and $y$ are finite paths in $E$, we observe that $C(y) \subseteq C(x)$, whenever $x$ is a subpath of $y$.
\end{Rmk}
\begin{equation} \label{eq:2009}
    C({\alpha}) \cap C({\beta})= 
\begin{cases}
    C({\beta}),& \text{if } \alpha \text{ is an initial subpath of } \beta \\
    C({\alpha}),& \text{if } \beta \text{ is an initial subpath of } \alpha \\
    \emptyset  ,& \text{otherwise}
\end{cases}
\end{equation}
Let $\varepsilon_v$ be the empty path starting at $v \in E^0$, then $\mathscr{G}_E^{(0)} = \partial E = \bigcup_{v \in E^0} C(\varepsilon_v) $.  This along with \ref{eq:2009} yields that the $C(\alpha)$ form a basis for the topology on $\partial E = \mathscr{G}_E^{(0)}$.  But this topology is not Hausdorff, since we cannot necessarily separate a finite path $\alpha$ ending at an infinite emitter from an infinite path with initial subpath $\alpha$.
\newline
This leads us to \textit{generalized cylinder sets} $C(\alpha, F)$ which are a basis for a Hausdorff topology for $\mathscr{G}_E^{(0)} = \partial E$ \cite[Section 2.2]{groupoidapproachleavitt}.  Let $\alpha \in E^{\ast}$, then $r(\alpha)E^1$ is the set of edges starting at $r(\alpha)$ and $F \subseteq r(\alpha)E^1$ is a finite subset of $r(\alpha)E^1$.
\begin{equation}
    C(\alpha, F) = C(\alpha) \setminus \bigcup_{e \in F} C(\alpha e)
\end{equation}
Note that when $F = \emptyset$, $C(\alpha, F) = C(\alpha)$, so the topology generated by the $C(\alpha, F)$ is a finer Hausdorff topology on $\mathscr{G}_E^{(0)} = \partial E$ than the topology generated by the $C(\alpha)$ \cite[Section 2.2, Lemma 2.3]{groupoidapproachleavitt}.  We also observe that the $C(\alpha)$ and $C(\alpha, F)$ both form compact open bases for $\partial E$ \cite[Section 2.2, Theorem 2.4 and Lemma 2.6]{groupoidapproachleavitt}.  So, equipped with the topology generated by the generalized cylinder sets, $ \mathscr{G}_E$ is a locally compact Hausdorff groupoid.
\newline 
If the graph, $E$, does not contain infinite emitters then the $C(\alpha)$ form a basis for the same Hausdorff topology on $\partial E$ as the $C(\alpha, F)$ generalized cylinder sets, allowing us to ignore the $C(\alpha, F)$.
\newline
A boundary path is \textit{eventually periodic} if it is of the form $p = \alpha \epsilon \epsilon \cdots \in E^{\infty}$ \cite{groupoidapproachleavitt}.
\begin{Prop} [Proposition 2.12 of \cite{groupoidapproachleavitt}] \label{thm:2007}If $E$ is a graph and $p \in \partial E$, then the isotropy group at $p$ is:
\begin{enumerate}
	\item infinite cyclic if $p$ is eventually periodic
	\item trivial if $p$ is not eventually periodic
\end{enumerate}
\end{Prop}
Note that the group ring of an infinite cyclic group for the coefficient ring $R$ is $R[x,x^{-1}]$.
\begin{Cor} \label{thm:2010}
Let $E$ be a graph and $p \in \partial E = R \mathscr{G}^{(0)}_E$
\begin{enumerate}
    \item $R (\mathscr{G}_E)_p^p$ is noetherian if and only if $R$ is noetherian.
    \item $R (\mathscr{G}_E)_p^p$ is artinian if and only if $R$ is artinian and p is not eventually periodic. 
\end{enumerate}
\end{Cor}
\begin{proof}
(1) follows immediately from Proposition \ref{thm:2007}, \textit{Connell's} Theorem \ref{thm:0011}, and Hilbert's basis Theorem \ref{thm:2011}.
\newline
(2) follows immediately from Proposition \ref{thm:2007} and \textit{Connell's} Theorem \ref{thm:0010}.
\end{proof}
\begin{Lemma} \label{thm:2013}
Let $E$ be a graph.  Then $\partial E$ will not contain any eventually periodic paths if and only if $E$ is acyclic.
\end{Lemma}
\begin{proof}
($\Rightarrow$)  Suppose $E$ has a cycle $\gamma$, then the path $\gamma \gamma \cdots \gamma \cdots$ is an eventually periodic path in $\partial E$.
\newline
($\Leftarrow$) Suppose $p = \alpha \gamma \gamma \cdots \gamma \cdots$ is an eventually periodic path in $\partial E$, then $\gamma$ is a closed path and must have a subpath that is a cycle.
\end{proof}
The following is a well known result and an immediate corollary to \cite[Lemma 2.7]{groupoidapproachleavitt}.
\begin{Cor} \label{thm:2001}
Suppose $\alpha$ is any finite path in $E$ then there exists an element of the boundary path space with $\alpha$ as a prefix, equivalently $C(\alpha) \ne \emptyset$.  
\end{Cor}
\begin{Lemma} \label{thm:2002}
Let $E = (E^0,E^1, s, r)$ be a directed graph, if the boundary path space $\partial E$ is discrete then $E$ does not have an infinite emitter.
\end{Lemma}
\begin{proof}
Suppose $E$ has an infinite emitter at the vertex $v \in E^0$.  Then there will certainly be an empty path $\varepsilon_v$ starting at $v$.  We will choose infinitely many boundary paths emanating from $v$ as $\gamma_i = e_i x_i \in \partial E$ where $e_i \in E^{1}, x_i \in \partial E$, $\{ e_i \}_{i \in I} \subseteq s^{-1}(v)$ for an infinite index set $I$ and $e_i \ne e_j$ for $i \ne j$ in $I$.  Now fix $k \in I$, then $\{ e_i x_i \}_{i \in I \setminus \{ k \}} \subseteq C(\varepsilon_v, e_k)$.  Therefore, by \eqref{eq:2009} and  Corollary \ref{thm:2001}, $C(\varepsilon_v, e_k)$ is an infinite compact set and the boundary path space cannot be discrete. 
\end{proof}
\begin{Thm} \label{thm:2003}
Let $E = (E^0,E^1, s, r)$ be a directed graph, then the boundary path space of $E$, $\partial E$, is discrete if and only if the 
\begin{enumerate}
    \item $E$ does not have infinite emitters: 
    \item every cycle in the graph has no exit; and
    \item every ray has only finitely many exits.
\end{enumerate}
\end{Thm}
\begin{proof}
For a row finite graph $E$, if $\eta \in \partial E$, then either
\begin{enumerate} [label=(\roman*)]
    \item \label{case:0001} $\eta$ has a vertex that is visited more than once --- this means that $\eta$ has a subpath that is a cycle; and if $\eta$ has a subpath $\gamma$ that is a cycle, then 
    \begin{enumerate}[label=(\alph*)]
        \item \label{case:0001-01} $\gamma$ has an exit; or
        \item \label{case:0001-02} $\gamma$ has no exits.
    \end{enumerate}
    \item \label{case:0002} none of the vertices in $\eta$ is visited more than once
    \begin{enumerate}[label=(\alph*)]
        \item \label{case:0002-01} $\eta$ has infinitely many exits; or 
        \item \label{case:0002-02} $\eta$ has finitely many exits, possibly zero 
        (this case also includes the case of $\eta$ ending in a sink).   
    \end{enumerate}    
\end{enumerate}
($\Rightarrow$) \newline
(1) Follows from Lemma \ref{thm:2002}.
\newline
By Lemma \ref{thm:2002}, we know that $E$ cannot have infinite emitters in order for $\partial E$ and therefore, $\mathscr{G}_E$ to be discrete.  For the remainder of this discussion we will assume that none of the elements of $\partial E$ end in an infinite emitter.  So we need only concern ourselves with infinite paths and paths ending in a sink.
\newline
Let $\gamma$ be a cycle with an exit, then $\gamma = e_1 \cdots e_n$ where $n$ is finite and $e_i \in E^1$ for $1 \le i \le n$ and the $e_i$ are distinct.  We can assume without loss of generality that our exit is at the vertex $v = s(e_1)=r(e_n)$.  So $|s^{-1}(v)|>1$, let $f \in E^1 \cap s^{-1}(v)$ and $f \ne e_1$.  So $f$ is not an edge in $\gamma$ and is therefore an exit from $\gamma$.  Let $\varepsilon_v$ be the empty path starting at $v$, then $\bigcup_{k \in \mathbb{N}} C(\gamma^k f) \subseteq C(\varepsilon_v)$.  By Corollary \ref{thm:2001} we know that the $C(\gamma^k f)$ are non-empty for all $k \in \mathbb{N}$ and $C(\gamma^j f)$, $C(\gamma^k f)$ are obviously disjoint for $j, k \in \mathbb{N}$, $j \ne k$.  So $C(\varepsilon_v)$ is an infinite compact set and therefore $\partial E$ cannot be discrete.
\newline
Let $\eta \in \partial E$ be a path with no cycles and infinitely many exits from the $\eta$ and let $\{ \alpha_i \}_{i \in I}$ be a set  of initial subpaths of $\eta$ such that $r(\alpha_i)$ has an exit $f_i$ from $\eta$ and $f_i \ne f_j$ whenever $i \ne j$.  By \eqref{eq:2009} and Corollary \ref{thm:2001}, each $C(\alpha_i f_i)$ is non-empty and  $C( \alpha_i f_i) \cap C( \alpha_j f_j) = \emptyset$ for $i,j \in I$ and $i \ne j$.  Let $v = s(\eta)$, then $\bigcup_{i \in I} C(\alpha_i f_i) \subseteq C(\varepsilon_v)$ and $C(\varepsilon_v)$ must be an infinite compact set.  Hence $\partial E$ cannot be discrete.
\newline
($\Leftarrow$)  Suppose a directed graph $E$ satisfies (1), (2), and (3), then the only types of boundary paths in $E$ are paths entering a cycle with no exit, finite paths ending in a sink, and rays with finitely many exits.
\newline
Let $\eta \in \partial E$ repeat a vertex, so $\eta$ will enter a cycle $\gamma$ and never leave by (2).  So $\eta$ will be of the form $\alpha \gamma \gamma \cdots \gamma \cdots$  where $\alpha$ is some finite path.  Then $C(\alpha \gamma)$ will contain only one element, the infinite path $\eta$.
\newline
Suppose $\eta \in \partial E$ has no repeated vertices.  Then either $|\eta| < \infty$, therefore ends in a sink or $|\eta| = \infty$ and is therefore a ray.
\newline
Subcase a: Suppose $\eta$ is an path ending in a sink, then $\eta = e_1 \cdots e_n$ for some finite $n$, where the $e_i$ are edges and $s^{-1}(r(e_n)) = \emptyset$.  Then $C(\eta) = C(e_1 \cdots e_n) =\{ \eta \}$.
\newline
Subcase b:  Suppose that $\eta \in \partial E$ is a ray with finitely many exits from the path $\eta$, where $n$ is finite.  Let $\gamma$ be an initial subpath of $\eta$  such that $r(\gamma)$ is past the finitely many exits from $\eta$.  Then $C(\gamma) = \{ \eta \}$.
\end{proof}
\begin{Rmk} \label{rmk:2002}
Observe that sinks do not affect whether $\partial E$ and $\mathscr{G}_E$ are discrete in Theorem \ref{thm:2003}.
\end{Rmk}
\begin{Cor} \label{thm:2016}
Let $R$ be a commutative ring with unit, $E$ a graph and $\mathscr{G}_E$ be the boundary path groupoid of $E$.  Then $R (\mathscr{G}_E)_p^p$ is artinian for all $p \in \partial E$ if and only if $R$ is artinian and $E$ is acyclic.
\end{Cor}
\begin{proof}
By Connell’s Theorem \ref {thm:0010}, Corollary \ref{thm:2010}, and Lemma \ref{thm:2013}, $R (\mathscr{G}_E)_p^p$ is artinian for all $p \in \partial E$ $\Leftrightarrow$ $R$ is artinian and $p$ is not eventually periodic for all $p \in \partial E$ $\Leftrightarrow$ $R$ is artinian and $E$ is acyclic.
\end{proof}
\begin{Lemma} \label{thm:2017}
Let $E$ be a graph satisfying the condition that cycles have no exits.
\begin{enumerate}
	\item Then every closed path, $\gamma \in E^{\ast}$, is of the form $\varepsilon^k$, where $\varepsilon$ is a cycle and $k \in \mathbb{N}$
	\item For all $p \in \partial E$, $p$ is eventually periodic if and only if $p$ ends in a cycle.
\end{enumerate}
\begin{proof}
(1) Every closed path $\gamma$ is a finite path that must have a subpath that is a cycle $\varepsilon$.  Since $E$ satisfies the no exits condition for cycles, $s(\gamma) = r(\gamma) = s(\varepsilon) = r(\varepsilon)$
(2)	($\Leftarrow$)  Follows by definition.  ($\Rightarrow$) By definition $p = \alpha \gamma \gamma \cdots \gamma \cdots $ where $\gamma$ is a closed path.  The rest follows from (1).
\end{proof}
\end{Lemma}
\begin{Rmk} \label{rmk:2003}
Under the identification of $\mathscr{G}^0_E$ with $\partial E$, an orbit $\mathcal{O} \subseteq \mathscr{G}^0_E$ is a collection of boundary paths in $\partial E$,  such that for any paths $p, q \in \mathcal{O}$, $p$ and $q$ will only differ by some finite initial subpath.  So we can write $p = ax$ and $q = bx$, where $a$ and $b$ are finite paths the $|a| - |b| \in \mathbb{Z}$ and $x \in \partial E$.
\end{Rmk}
We will use the well known fact that $L_R(E) = R \mathscr{G}_E$ and our prior results to show that our main Theorem \ref{thm:1001} recovers Theorems 4.2.7 and 4.2.12 of  \cite{LeavittBook}.
\begin{Thm} \label{thm:2014}  Let $E$ be a graph.  Let $R$ be a commutative ring with unit.  Then the following conditions are equivalent
\begin{enumerate}
    \item $L_R(E)$ \label{thm:2014_01} is categorically left noetherian.
    \item $L_R(E)$ \label{thm:2014_02} is categorically right noetherian.
    \item \label{thm:2014_03} $L_R(E)$ is locally left noetherian.
    \item \label{thm:2014_04} $L_R(E)$ is locally right noetherian.
    \item \label{thm:2014_05} $R$ is noetherian and 
    \begin{enumerate}
        \item $E$ does not have infinite emitters: 
        \item every cycle in the graph has no exit; and
        \item every ray has only finitely many exits.
    \end{enumerate}
    \item \label{thm:2014_06} $L_R(E) \cong \bigoplus _{i \in J_1} M_{\mathscr{O}_{p_i}}(R) \oplus \bigoplus _{j \in J_2} M_{\mathscr{O}_{p_j}}(R[x,x^{-1}])$ where the $p_i$ with $i \in J_1$ represent the orbits of non-eventually periodic paths and $p_j$ with $j \in J_2$ represent the orbits of eventually periodic paths and $R$ is noetherian.
\end{enumerate}
\end{Thm}
\begin{proof}
We will show the equivalence between each of Theorem \ref{thm:2014} conditions  (\ref{thm:2014_01})-(\ref{thm:2014_06}) and a condition in Theorem \ref{thm:1001} and the results follows immediately.  Since $L_R(E) = R \mathscr{G}_E$, Theorem \ref{thm:2014}(\ref{thm:2014_01}) and (\ref{thm:2014_02}) are obviously equivalent to Theorem \ref{thm:1001}(\ref{thm:1001_01}), and Theorem \ref{thm:2014}(\ref{thm:2014_03}) and (\ref{thm:2014_04}) are equivalent to Theorem \ref{thm:1001}(\ref{thm:1001_02}).
\newline 
All that remains is to show the translation between Theorem \ref{thm:2014}(\ref{thm:2014_05}) and (\ref{thm:2014_06}) to the appropriate language of etale groupoid algebras in Theorem \ref{thm:1001}.  When $L_R(E) = R \mathscr{G}_E$, Corollary \ref{thm:2010} and tells us that $R (\mathscr{G}_E)_p^p$ is always noetherian if and only if $R$ is noetherian.  By the preceding and our characterization of discreteness for $\mathscr{G}_E$ in Theorem 
\ref{thm:2003}, Theorem \ref{thm:1001}(\ref{thm:1001_03}) is equivalent to $R$ is noetherian and 
\begin{enumerate}[label=(\alph*)]
    \item $E$ does not have infinite emitters;
    \item every cycle in the graph has no exits; and
    \item every ray has only finitely many exits.
\end{enumerate}
As a result of Proposition \ref{thm:2007}, we can split the index set of orbits $J$ into a disjoint union $J_1$ and $J_2$, where $J_1$ and $J_2$ are defined above, and $R(\mathscr{G}_E)_{p_i}^{p_i} = R$ for $i \in J_1$ (i.e., when $p_i$ is not an eventually periodic path) and $R(\mathscr{G}_E)_{p_j}^{p_j} = R[x,x^{-1}]$ for $j \in J_2$ (i.e., when $p_j$ is an eventually periodic path). 
\newline
By the preceding and Corollary \ref{thm:2010}, Theorem \ref{thm:1001}(\ref{thm:1001_04}) yields $R$ is noetherian and 
\begin{equation} \label{eq:2014}
\begin{split}
L_R(E) \cong R \mathscr{G} & \cong  \bigoplus _{i \in J_1} M_{\mathscr{O}_{p_i}}(R (\mathcal{G}_E)_{p_i}^{p_i}) \oplus \bigoplus _{j \in J_2} M_{\mathscr{O}_{p_j}}(R (\mathcal{G}_E)_{p_j}^{p_j}) \\ & =\bigoplus _{i \in J_1} M_{\mathscr{O}_{p_i}}(R) \oplus \bigoplus _{j \in J_2} M_{\mathscr{O}_{p_j}}(R[x,x^{-1}]).
\end{split}
\end{equation}
\end{proof}
\begin{Thm} \label{thm:2015} Let $E$ be a graph and $R$ be a commutative ring with unit.  The following conditions are equivalent:
\begin{enumerate}
    \item \label{thm:2015_01} $L_R(E)$ is categorically left artinian.
    \item \label{thm:2015_02} $L_R(E)$ is categorically right artinian.
    \item \label{thm:2015_03} $L_R(E)$ is locally left artinian.
    \item \label{thm:2015_04} $L_R(E)$ is locally right artinian.
    \item \label{thm:2015_05} $R$ is artinian and
    \begin{enumerate}
        \item $E$ does not have infinite emitters: 
        \item every ray has only finitely many exits; and
        \item $E$ is acyclic 
    \end{enumerate}
    \item \label{thm:2015_06} $L_R(E) \cong \bigoplus _{\alpha \in J} M_{\mathscr{O}_{\alpha}}(R)$ where $J$ is the index set for the set of orbits and $R$ is artinian.
\end{enumerate}
\end{Thm}
\begin{proof}
As before we will show the equivalence between each of Theorem \ref{thm:2015} conditions  (\ref{thm:2015_01})-(\ref{thm:2015_06}) and a condition in Theorem \ref{thm:1001} and the results follows.  Since $L_R(E) \cong R \mathscr{G}_E$, the equivalence of Theorem \ref{thm:2015}(\ref{thm:2015_01})-(\ref{thm:2015_04}) to Theorem \ref{thm:1001}(\ref{thm:1001_01}) and (\ref{thm:1001_02}) is clear.
\newline
We now show the translation between Theorem \ref{thm:2015}(\ref{thm:2015_05}) and (\ref{thm:2015_06}) to the appropriate language of etale groupoid algebras in Theorem \ref{thm:1001}.  Corollary \ref{thm:2010} and Lemma \ref{thm:2013} tell us that $R (\mathscr{G}_E)_p^p$ is artinian for all $p \in \partial E$ if and only if $R$ is artinian and $E$ is acyclic.  By the preceding and our characterization of discreteness for $\mathscr{G}_E$ in Theorem 
\ref{thm:2003}, Theorem \ref{thm:1001}(\ref{thm:1001_03}) is equivalent to $R$ is artinian and 
\begin{enumerate}[label=(\alph*)]
    \item $E$ does not have infinite emitters;
    \item every ray has only finitely many exits
    \item $E$ is acyclic.
\end{enumerate}
By Proposition \ref{thm:2007} and Corollary \ref{thm:2010}, 
$R(\mathscr{G}_E)_{p_i}^{p_i}$ is artinian for all $p_i \in \partial E$ if and only if $R$ is artinian and $p_i$ is not eventually periodic for all $p_i \in \partial E$  if and only if $R$ is artinian and $R(\mathscr{G}_E)_{p_i}^{p_i} = R$ for all $i \in J$. 
Therefore Theorem \ref{thm:1001}(\ref{thm:1001_04}) is equivalent to $R$ is artinian and 
\begin{equation}
L_R(E) \cong R \mathscr{G} =\bigoplus _{i \in J} M_{\mathscr{O}_i}(R).
\end{equation}
\end{proof}
\begin{Rmk} \label{rmk:2004}
Theorems 4.2.12 and 4.2.7 of \cite{LeavittBook} assume the coefficient ring $R$ is a field and we show in Theorems \ref{thm:2014} and \ref{thm:2015} that these results holds in the more general setting where $R$ is any unital commutative ring.
When comparing our generalization of the categorical and local chain conditions results in \cite{LeavittBook}, the phrase \textit{``infinite paths ending in a sink"} in \cite{LeavittBook} is the same as our phrase \textit{``ray with finitely many exits"} in the presence of cycles having no exits.  
\end{Rmk}
\begin{Thm} \label{thm:2021} Let $E$ be a graph and $R$ be a commutative ring with unit, then the following conditions are equivalent:
\begin{enumerate}
  \item \label{thm:2021_01} $L_R(E)$ is semisimple. 
  \item \label{thm:2021_02} $L_R(E)$ is locally semisimple.
  \item \label{thm:2021_03} $R$ is semisimple and
      \begin{enumerate}
        \item \label{thm:2021_03_a} $E$ does not have infinite emitters: 
        \item \label{thm:2021_03_b} every cycle in the graph has no exit; and
        \item \label{thm:2021_03_c} E is acyclic.
    \end{enumerate}
  \item \label{thm:2021_04} $L_R(E) \cong \bigoplus _{\alpha \in J} M_{\mathscr{O}_{\alpha}}(R)$ where $J$ is the index set for the set of orbits and $R$ is semisimple. 
\end{enumerate}
\end{Thm}
\begin{proof}
(\ref{thm:2021_01}) $\Leftrightarrow$ (\ref{thm:2021_02}) As in Theorem \ref{thm:2015} \ref{thm:2015_01}-\ref{thm:2015_06}, the equivalence of \ref{thm:2021} (\ref{thm:2021_01}) and (\ref{thm:2021_02}) follows from $L_R(E) \cong R \mathscr{G}_E$ and Theorem \ref{thm:1012}.
\newline
(\ref{thm:2021_02}) $\Rightarrow$ (\ref{thm:2021_03})  Follows from an appropriate translation of Theorem \ref{thm:1012} (\ref{thm:1012_02}) and (\ref{thm:1012_03}).  For any unital ring $R$ and group $G$, there is a surjective ring homomorphism $\varphi: RG \rightarrow R$ sending $rg \mapsto r$ for $r \in R$ and $g \in G$.  Any surjective ring homomorphism $\varphi$ from a semisimple ring forces the image of $\varphi$ to be semisimple, and therefore $R$ is semisimple.  By assumption, $L_R(E)$ is locally semisimple and, for unital rings, semisimple implies artinian; hence, locally semisimple implies locally artinian.  Therefore, $L_R(E)$ is locally semisimple and by Theorem \ref{thm:2015} we know that $R (\mathscr{G}_E)_p^p$ must be artinian for all $p \in \partial E$ and $E$ must satisfy the following
    \begin{enumerate}
        \item $E$ does not have infinite emitters: 
        \item every ray has only finitely many exits; and
        \item $E$ is acyclic 
    \end{enumerate}
(\ref{thm:2021_03}) $\Rightarrow$ (\ref{thm:2021_04}) follows from Theorem \ref{thm:2015} and the semisimplicity of $R$.
\newline
(\ref{thm:2021_04}) $\Rightarrow$ (\ref{thm:2021_01}) follows from Propositions \ref{thm:1014} and \ref{thm:1013}.
\end{proof}
\subsection{Future work}
This is a preliminary work, so we will post a forthcoming update, which will apply our results to higher rank graph algebras, specifically the Kumjian-Pask algebras, and inverse semigroup algebras.
\def\malce{\mathbin{\hbox{$\bigcirc$\rlap{\kern-7.75pt\raise0,50pt\hbox{${\tt m}$}}}}}\def\cprime{$'$} \def\cprime{$'$} \def\cprime{$'$} \def\cprime{$'$} \def\cprime{$'$} \def\cprime{$'$} \def\cprime{$'$} \def\cprime{$'$} \def\cprime{$'$} \def\cprime{$'$}

\end{document}